\documentclass[11pt,reqno]{amsart}
\usepackage{amscd,graphicx,amsxtra,amsmath,color}

\usepackage[frame,cmtip,arrow,matrix,line,graph,curve]{xy}

\usepackage[colorlinks]{hyperref}
\usepackage{theoremref}

\newcommand{\C}{\mathbb C}
\newcommand{\X}{\mathcal X}
\newcommand{\M}{\mathcal M}

\newcommand{\ZZ}{\mathbb Z}

\renewcommand{\phi}{\varphi}

\newcommand{\SW}{\operatorname{SW}\,}

\newcommand{\rk}{\operatorname{rank}}

\newcommand{\CP}{{\mathbb C}{\rm P}}

\newcommand{\sign}{\operatorname{sign}}

\newcommand{\PP}{\mathbb{P}}

\newcommand{\Dir}{\,\operatorname{Dir}}
\newcommand{\Sign}{\,\operatorname{Sign}}
\newcommand{\Spec}{\,\operatorname{Spec}}

\newcommand{\spinc}{\ifmmode{\operatorname{Spin}^c}\else{$\operatorname{spin}^c$\ }\fi}

\newcommand{\lt}{\left(\hspace{-0.07in}\left(}
\newcommand{\rt}{\right)\hspace{-0.07in}\right)}
\newcommand{\mmod}{\hspace{-0.08in}\mod}

\newtheorem{theorem}{Theorem}[section]

\newtheorem{lemma}[theorem]{Lemma}
\newtheorem{proposition}[theorem]{Proposition}
\newtheorem{corollary}[theorem]{Corollary}
\theoremstyle{definition}

\newtheorem{remark}[theorem]{Remark}

\title{Instanton Floer homology and Milnor fibers}
\thanks{The first author was partially supported by the Simons Investigator Award-HMS (PI Ludmil Katzarkov) and the University of Miami. The third author was partially supported by the NSF Grant DMS-1952762}
\author[Kyoung-Seog Lee]{Kyoung-Seog Lee}
\address{Department of Mathematics \newline\indent 
POSTECH \newline\indent 77 Cheongam-ro, Nam-gu, Pohang, Gyeongbuk, 37673, Korea}
\email{\rm{kyoungseog@postech.ac.kr}}
\author[Anatoly Libgober]{Anatoly Libgober}
\address{Department of Mathematics \newline\indent 
University of Illinois, Chicago \newline\indent Chicago, IL}
\email{\rm{libgober@uic.edu}}
\author[Nikolai Saveliev]{Nikolai Saveliev}
\address{Department of Mathematics\newline\indent
University of Miami \newline\indent PO Box 249085
\newline\indent Coral Gables, FL 33124}
\email{\rm{saveliev@math.miami.edu}}

\begin{document}
\begin{abstract}{In the early days of the Floer theory, Atiyah asked if there is a Milnor fiber description of the Floer homology of the links of singularities. We answer this question for the Brieskorn–Hamm complete intersection singularities. The resulting combinatorial formulas lead to an independent proof of the equality of the Casson invariants in the Donaldson and Seiberg--Witten theories. We use similar techniques to express the Heegaard Floer $d$-invariant of torus knots in terms of their Milnor fibers.}
\end{abstract}

\maketitle

The following question, which appears in the paper of Neumann and Wahl \cite[Question 3.5]{Neumann-Wahl}, is attributed to Michael Atiyah:

\bigskip

\begin{center}

\begin{minipage}{33em}
\emph{Is there a Milnor fiber description of the Floer homology of the link? For $\Sigma(p, q, r)$, is it related to the action of complex conjugation on the homology of the Milnor fiber of $x^p +y^q + z^r = 1$?}
\end{minipage}

\end{center}

\bigskip
The Brieskorn homology spheres $\Sigma(p,q,r)$ referred to in this question are the links of singularity at zero of the complex polynomial $x^p + y^q + z^r$, for all pairwise relatively prime integers $p$, $q$ and $r$ greater than or equal to 2. By the Floer homology one means the instanton Floer homology of integral homology spheres as originally defined by Floer \cite{Floer}. 

\thispagestyle{empty}

The same question can be asked more generally for Seifert fibered homology spheres $\Sigma (a_1,\ldots, a_n)$ which are the links of the Brieskorn--Hamm complete intersection singularities 
\begin{equation}\label{E:hamm}
\{\,c_{i1} z_1^{a_1} + \ldots + c_{in} z_n^{a_n} = 0,\; i=1,\ldots,n-2\,\}\, \subset\, \mathbb C^n,
\end{equation}
where $a_1,\ldots, a_n$ are pairwise co-prime integers greater than or equal to 2 and $(c_{ij})$ is any real matrix of size $(n-2) \times n$ each of whose maximal minors is non-zero; see Hamm \cite{Hamm}. The case of $n = 3$ corresponds to the Brieskorn homology spheres $\Sigma(p,q,r)$.

In the late 1990s, the third named author \cite{Saveliev} came up with a closed form formula for the Floer homology of $\Sigma(a_1,\ldots,a_n)$, which could be viewed as an answer to Atiyah's question for the Brieskorn--Hamm singularities. The formula, which will be recalled in Section \ref{S:classical}, involves the Milnor fiber of the singularity and the complex conjugation on its link. In this paper, we wish to tie that formula more closely with the algebraic geometry of the Milnor fiber and with the progress in gauge theory of the last quarter century. We obtain two more formulas, one in terms of the compactification of the Milnor fiber in a weighted projective space in Section \ref{S:milnor} (Theorem \ref{T:euler} and Theorem \ref{T:bb}), and the other in terms of the spectrum of singularity in Section \ref{S:spectrum} (Theorem \ref{T:spectrum}). Section \ref{S:sw} (Theorem \ref{T:moy}) relates our formulas to Seiberg--Witten monopoles on $\Sigma(a_1,\ldots, a_n)$, which leads to an independent proof of the Witten--style conjecture, proposed by Mrowka, Ruberman and Saveliev \cite{MRS}, in the special case of Brieskorn--Hamm singularities. A generalization of these results to other links of singularities is discussed in Section \ref{S:comments}. A new formula for the $\bar\mu$--invariant of Neumann and Siebenmann is derived in Section \ref{S:mubar} as part of the proof of Theorem \ref{T:bb}. 

In Section \ref{S:knots}, we derive several formulas for the $d$--invariant of torus knots in terms of Milnor fibers, theta characteristics, and the Steenbrink Jacobian. While the $d$--invariant arises in the Heegaard Floer and not instanton homology, we included these results here because of their thematic similarity to the other results of the paper.

\medskip\noindent
\textbf{Acknowledgments:}\; Part of this work was done while the first author was an IMSA Research Assistant Professor at the University of Miami. He and the second named author thank Ludmil Katzarkov and Josef Svoboda for helpful discussions about related topics.


\section{Classical results}\label{S:classical}
We begin by reviewing the relevant results from the 1990s. The instanton homology groups $I_* (\Sigma (a_1,\ldots, a_n))$ are free abelian groups whose instanton Betti numbers 
\[
b_k : = \, \rk I_k (\Sigma(a_1,\ldots,a_n)),\quad k = 0, 1,\ldots, 7,
\]
have the property that $b_k = 0$ for even $k$. This was established by Fintushel and Stern\footnote{Fintushel and Stern work with self-dual rather than anti-self-dual equations hence their Floer grading matches ours after switching from $n$ to $-3-n$.} \cite{FS} for $n = 3, 4$ and proved for all $n$ by Kirk and Klassen \cite{KK}. Fintushel and Stern also produced an algorithm for computing $b_k$ for each individual $\Sigma (a_1,\ldots, a_n)$ with $n = 3, 4$. Neither this algorithm nor its extensions to arbitrary $n$ lead to a closed form formula for the Betti numbers $b_k$, hence we are taking a different approach.

Taubes \cite{Taubes} proved that the Euler characteristic of the instanton Floer homology of an arbitrary integral homology sphere $\Sigma$ equals twice the Casson invariant $\lambda(\Sigma)$. In particular, it follows that
\begin{equation}\label{E:casson}
-b_1 - b_3 - b_5 - b_7\; =\; 2\, \lambda(\Sigma(a_1,\ldots, a_n)).
\end{equation}
Saveliev \cite{Saveliev, Saveliev2} computed the Lefschetz number of the map induced on the instanton homology of $\Sigma(a_1,\ldots, a_n)$ by the complex conjugation involution. He obtained the formula
\begin{equation}\label{E:lefschetz}
-b_1 + b_3 - b_5 + b_7\; =\; 2\, \bar\mu(\Sigma(a_1,\ldots, a_n)),
\end{equation}
where $\bar\mu$ stands for the Neumann--Siebenmann invariant; see Neumann \cite{Neumann} and Siebenmann \cite{Siebenmann}. Recall that, for any plumbed homology sphere $\Sigma$, such as $\Sigma(a_1,\ldots, a_n)$,
\[
\bar\mu(\Sigma) = \frac 1 8\, (\,\sign P - w\cdot w),
\]
where $P$ is any plumbed manifold with oriented boundary $\partial P = \Sigma$ and $w$ is its integral Wu class, defined as the unique homology class in $H_2 (P)$ satisfying the following two conditions. First, $w$ is characteristic, that is (dot represents intersection number) $w\cdot x = x\cdot x \pmod 2$ for all $x\in H_2 (P)$, and second, all the coordinates of $w$ are either 0 or 1 in the natural basis of $H_2 (P)$ represented by embedded 2-spheres. 

The formulas \eqref{E:casson} and \eqref{E:lefschetz} together lead to 
\begin{align}
b_1 + b_5\; &= -\bar\mu (\Sigma(a_1,\ldots, a_n)) -  \lambda(\Sigma(a_1,\ldots, a_n)), \label{E:s1}
\\
b_3 + b_7\; &=\; \;\; \bar\mu (\Sigma(a_1,\ldots, a_n)) -  \lambda(\Sigma(a_1,\ldots, a_n)). \label{E:s2}
\end{align}
To finish the calculation of $I_*(\Sigma(a_1,\ldots, a_n))$, one uses the 4-periodicity in instanton homology proved by Fr{\o}yshov \cite{Froyshov},
\[
b_1 = b_5\quad\text{and}\quad b_3 = b_7.
\]


\section{Compactification of the Milnor fiber}\label{S:milnor}
Let us consider $n-2$ polynomials $P_i(z) = c_{i1} z_1^{a_1} + \ldots + c_{in} z_n^{a_n}$, where $c_{ij}$ are as in formula \eqref{E:hamm}, and the map $f: \mathbb C^n \to \mathbb C^{n-2}$ given by 
\[
f(z)\, = ( P_1(z), \cdots, P_{n-2}(z)). 
\]
The Milnor fiber $M(a_1,\ldots, a_n)$ of the singularity $f^{-1} (0)$ is the intersection of any fiber $f^{-1} (\delta)$ that is non-singular with the closed unit ball in $\mathbb C^n$. It is a simply-connected smooth compact spin 4-manifold with boundary $\Sigma(a_1,\ldots, a_n)$. 

To relate the above formulas for the instanton homology of $\Sigma(a_1,\ldots, a_n)$ to the Milnor fiber $M(a_1,\ldots, a_n)$, we use the formula 
\begin{equation}\label{E:fs}
\lambda(\Sigma(a_1,\ldots, a_n))\;=\;\frac 1 8\,\sign M(a_1,\ldots, a_n),
\end{equation}
which was proved for $n = 3$ by Fintushel and Stern \cite{FS} and for all $n \ge 4$ by Neumann and Wahl \cite{Neumann-Wahl}. Both of these proofs are combinatorial but there is also a geometric proof by Collin and Saveliev \cite{CS1, CS2}. Formula \eqref{E:fs} expresses, roughly speaking, half of the information contained in $I_*(\Sigma(a_1,\ldots, a_n))$ in terms of the Milnor fiber. The other half can be recovered, for instance, by expressing the quantity $b_3 + b_7$ in terms of the Milnor fiber. One way to accomplish this is as follows. 

Let us compactify the Milnor fiber $M = M(a_1,\ldots,a_n)$ in the weighted projective space $\PP = \CP^n (q_1: \ldots : q_n : 1)$, where $q_i = a/a_i$ and $a = a_1\cdots a_n$. We obtain a singular algebraic surface $X$ with cyclic quotient singularities. Its minimal resolution $Y \to X$ is a smooth algebraic surface of the form $Y =  M \cup P$, where $P$ is a plumbed manifold with the boundary $-\Sigma (a_1,\ldots, a_n)$. In fact, for $\Sigma(2,3,7)$, the manifold $Y$ is the Kummer K3--surface. For $\Sigma (2,3,6k \pm 1)$, it is an elliptic surface. Otherwise, $Y$ is a surface of general type.

The quantity $b_3 + b_7$ can be expressed in terms of the holomorphic Euler characteristic of $Y$ as follows. Let $K_Y$ be the canonical divisor of $Y$ and extend the integral Wu class $w \in H_2 (P)$ by zero to a homology class on $Y$; call it again $w$. Since $K_Y$ and $w$ are both characteristic, the classes $K_Y \pm w$ are divisible by two. We will use the notation
\[
K_Y + w\, =\, 2\left\lceil {K_Y}/2 \right\rceil\quad\text{and}\quad K_Y - w\, =\, 2\left\lfloor {K_Y}/2 \right\rfloor.
\]

\begin{theorem}\label{T:euler} 
Let $K_Y$ be the canonical divisor of $Y$. Then
\[
b_3 \, +\, b_7\; = \; \chi \left(Y, \mathcal O\left(\left\lceil {K_Y}/2 \right\rceil\right) \right)\; =\; \chi \left(Y, \mathcal O \left(\left\lfloor {K_Y}/2 \right\rfloor\right) \right).
\]
\end{theorem}

\begin{proof}
Let us temporarily use notation $K_Y + w = 2D$. It follows from the Hirzebruch--Riemann--Roch theorem that
\[
\chi \left(Y,\mathcal O(D)\right)\, =\, \chi (Y)\, +\, \frac 1 2\, D\cdot (D-K_Y)\, =\, \chi (Y) - \frac 1 8\, (K_Y\cdot K_Y - w\cdot w).
\]
Next, use Noether's formula 
\[
\chi (Y)\, =\, \frac 1 {12}\, \left(\,c_1^2(Y) + c_2(Y)\right)[Y]
\]
and the universal formula $p_1 = c_1^2 -2c_2$ for the first Pontryagin class to deduce that
\[
\chi (Y)\, =\, - \frac 1 {24}\; p_1 (Y) [Y]\, +\, \frac 1 8\; c_1^2 (Y) [Y]\, =\, -\frac 18\,\left(\sign\, (Y) - K_Y\cdot K_Y\right),
\]
with the last equality following from the Hirzebruch signature theorem. This shows that
\[
\chi \left(Y,\mathcal O(D)\right)\, = -\frac 1 8\,(\sign (Y) - w\cdot w) = - \lambda(\Sigma(a_1,\ldots, a_n)) + \bar\mu (\Sigma(a_1,\ldots, a_n)) = b_3 + b_7,
\]
which proves the first formula of the theorem. The second formula is obtained by replacing $w$ with $-w$ in the above argument.
\end{proof}

\begin{remark} It is worth noting that the compactification of a Milnor fiber in a weighted projective space also played a prominent role in the paper of Ebeling and Okonek \cite{EO}. They employed it to compute relative Donaldson polynomials of the Milnor fibers of Brieskorn--Hamm complete intersection singularities by using the TQFT properties of the Donaldson theory. It would be interesting to relate their work to ours. 
\end{remark}

We wish next to express the quantity $b_3 + b_7$ directly in terms of the singular surface $X$, without first resolving its singularities. This will allow us to obtain a formula for $b_3 + b_7$ in terms of the spectrum of singularity in Section \ref{S:spectrum}. 

\begin{theorem}\label{T:bb}
Let $K_X$ be the canonical divisor of $X$. Then 
\[
b_3 + b_7\; =\; \chi \left(X, \mathcal O\left(\left\lceil {K_X}/2 \right\rceil\right) \right)\; = \; \chi \left(X, \mathcal O\left(\left\lfloor {K_X}/2 \right\rfloor\right) \right).
\]
\end{theorem}

The proof of this result, which is based on Kawasaki's version of the Atiyah--Singer index theorem, is rather technical and we postpone it until Section \ref{S:mubar} for the sake of exposition. We hope to present an algebraic geometric proof elsewhere. 


\section{Instanton homology and integer lattice points}\label{S:spectrum}
In this section, we will identify the holomorphic Euler characteristics appearing in Theorem \ref{T:bb} with a   count of integer lattice points in a simplex. This will lead to an explicit combinatorial formula for $b_3 + b_7$. 


\subsection{The statement} 
Given a complete intersection singularity \eqref{E:hamm}, for any integer $e \ge 0$ denote by $A_e$ the set of $n$--tuples $(x_1,\ldots, x_n)$ of integers such that $0 \le x_i < a_i$ and 
\medskip
\[
e\, +\, \sum_{i = 1}^n \; \frac {x_i}{a_i} \; < \; \frac 1 2\, \left(n - 2\, -\, \sum_{i = 1}^n\; \frac 1 {a_i} \right).
\]

\medskip\noindent
One can easily see that the above inequality is only possible when $e \le \lfloor (n - 3)/2 \rfloor$. 

\begin{theorem}\label{T:spectrum}
Let $\big|A_e \big|$ denote the cardinality of the set $A_e$ and let $s =  \lfloor (n - 3)/2 \rfloor$. Then 
\[
b_3 + b_7\; = \; 2\, \sum_{e = 0}^s\; (e + 1)\, \big| A_e \big|.
\]
\end{theorem}

\noindent
Before we go on to prove this theorem, we will explain its relation to the spectrum of singularity and spell out its statement in the special case of $n = 3$, where the resulting formulas are particularly simple. 


\subsection{Brieskorn singularities} 
Associated with the singularity at zero of the polynomial $f(x,y,z) = x^p + y^q + z^r$ is its spectrum 
\medskip
\[
\Spec_f\, = \, \left\{\;\frac k p + \frac \ell q + \frac m r \;\; \Big|\;\; 0 < k < p,\;\; 0 < \ell < q,\;\; 0 < m < r\;\right\}.
\]

\medskip\noindent
Recall that $\Spec_f$ is the set of rational numbers characterized by the property that the complex numbers $\exp ( 2\pi i (k/p + \ell/q + m/r))$ form a complete set of  eigenvalues of the monodromy operator acting on the second homology of the Milnor fiber $M(p,q,r)$, and that their integral parts correspond to the Hodge type of the corresponding eigenvectors; see Steenbrink \cite{Steenbrink} for details. In our case, the eigenvalues are simple and there are exactly $(p-1)(q-1)(r-1)$ of them, the Milnor number of the singularity. The following two theorems express the instanton homology of $\Sigma (p,q,r)$ in terms of this spectrum. 

\begin{theorem}\label{T:sign}
The quantity $b_1 + b_3 + b_5 + b_7$ equals 1/4 times $- \tau_1 + \tau_2 - \tau_3$, where $\tau_n$ is the number of points in $\Spec_f$ such that
\[
n - 1\; <\;  \frac k p\, +\, \frac \ell q\, +\, \frac m r \; < \; n.
\]
\end{theorem}

\begin{proof}
The result follows easily by combining formulas \eqref{E:casson} and \eqref{E:fs} with the Brieskorn formula \cite{Brieskorn} for the Milnor fiber signature; see also Steenbrink \cite{Steenbrink1}.
\end{proof}

\begin{theorem}\label{T:sw}
The quantity $b_3 + b_7$ equals twice the number of the points in $\Spec_f$ with the property that
\begin{equation}\label{E:ineq}
 \frac k p\, +\, \frac \ell q\, +\, \frac m r \; <\; \frac 1 2\, \left(1\, +\, \frac 1 p\, +\, \frac 1 q\, +\, \frac 1 r \right).
\end{equation}
\end{theorem}

\smallskip

\begin{proof}
Theorem \ref{T:spectrum} implies that $b_3 + b_7 = 2\, | A_0 |$, where the set $A_0$ consists of the triples $(x_1,x_2,x_3)$ of integers such that\, $0 \le x_1 < p$,\; $0\le x_2 < q$,\; $0\le x_3 < r$\, and 
\medskip
\[
\frac {x_1} p\, +\, \frac {x_2} q\, +\, \frac {x_3} r \; < \; \frac 1 2\, \left(1\, -\, \frac 1 p\, -\, \frac 1 q\, -\, \frac 1 r \right).
\]

\medskip\noindent
Substituting $k = x_1 + 1$, $\ell = x_2 + 1$ and $m = x_3 + 1$ results in inequality \eqref{E:ineq}. The result now follows because \eqref{E:ineq} automatically implies that $k < p$,\, $\ell < q$ and $m < r$.
\end{proof}

\begin{remark}
Singularities \eqref{E:hamm} with homology sphere links have been studied by Hamm \cite{Hamm} for all $n \ge 4$. While the eigenvalues of the monodromy operator are always of the form $\exp(2\pi i (\ell_1/a_1 + \ldots + \ell_n/a_n))$, for non-negative integers $\ell_1,\ldots, \ell_n$, their multiplicities are generally greater than one, and their precise relation to the formula of Theorem \ref{T:spectrum} remains unclear. In addition, we are not aware of an analogue of the Brieskorn formula for the Milnor fiber signature when $n \ge 4$. 
\end{remark}


\subsection{Proof of Theorem \ref{T:spectrum}}\label{S:3.3}
As before, let $X$ be the compactification of the Milnor fiber $M(a_1,\ldots,a_n)$ in the weighted projective space $\PP = \C \PP^n (q_1: \ldots : q_n : 1)$, where $q_i = a/a_i$ and $a = a_1\cdots a_n$. Specifically, $X$ is an algebraic surface which is a complete intersection given by the equations
\[
c_{i1} z_1^{a_1} + \ldots + c_{in} z_n^{a_n}\, =\, \delta_i z_0^a, \quad i = 1,\cdots, n-2,
\]
for generic $\delta_1,\ldots, \delta_{n-2}$. It has cyclic quotient singularities on the hyperplane $z_0 = 0$ and its canonical class is given by
\[
K_X\, =\, a \left( n-2 \, - \, \sum_{i=1}^n \; \frac 1 {a_i} \, -\, \frac 1 a\right)H,
\]

\smallskip\noindent
where $H$ is the the pullback to $X$ of the positive generator of the class group of  Weil divisors on $\PP$.  Consider the Koszul resolution of the structure sheaf of this complete intersection: 
\[
0 \rightarrow \Lambda^{n-2} (\mathcal O_{\PP}(-a)^{n-2}) \rightarrow \cdots \rightarrow \Lambda^e \left(\mathcal O_{\PP}(-a)^{n-2}\right)\rightarrow \cdots \rightarrow \mathcal O_{\PP} \rightarrow \mathcal O_X \rightarrow 0,
\]
where 
\[
\Lambda^e (\mathcal O_{\PP}(-a)^{n-2})\, =\, {{n-2}\choose {e}}\; \mathcal O_{\PP}(-ea).
\]

\smallskip\noindent
Tensoring this resolution with ${\mathcal O}_{\PP} \left(\lfloor K_X/2 \rfloor\right)$ yields the sequence
\begin{align}
0 \rightarrow \Lambda^{n-2} ({\mathcal O}_{\PP}(-a)^{n-2}) & \otimes {\mathcal O}_{\PP} \left(\lfloor K_X/2 \rfloor\right) \rightarrow \notag \\ \cdots & \rightarrow \Lambda^e \left(\mathcal O_{\PP}(-a)^{n-2}\right) \otimes {\mathcal O}_{\PP} \left(\lfloor K_X/2 \rfloor\right) \rightarrow \cdots \label{E:star} \\
 &\qquad\qquad\qquad  \rightarrow {\mathcal O}_{\PP} \left(\lfloor K_X/2 \rfloor\right) \rightarrow {\mathcal O}_{X} \left(\lfloor K_X/2 \rfloor\right) \rightarrow 0. \notag
\end{align}

This sequence is exact, which can be seen as follows \footnote{\,The sheaves $\lfloor K_X/2 \rfloor$ and $\lceil K_X/2 \rceil$ are reflexive but not necessarily locally free.}. Recall that $\PP$ is a global quotient $\pi: \CP^n \rightarrow \PP$ by the action of a finite cyclic group $G$; see Dolgachev \cite{Dolgachev}. Following Canonaco \cite{Canonaco}, consider the functor $\pi_*^G$ assigning to a coherent sheaf $\mathcal F$ on  $\PP$ its component  in the character decomposition 
\[
\pi_*({\mathcal F})=\bigoplus_{\chi\, \in\, {\rm Char}(G)} \pi_*^{\chi}({\mathcal F})
\]
corresponding to the trivial character $\chi=1$. Then $X = X'/G$, where $X$ is a complete intersection in $\CP^n$, and 
$
\pi_*^G({\mathcal O}_{\CP^n}(i)) = {\mathcal O}_{\PP}(i)
$
by \cite[Corollary 1.4]{Canonaco}. The tensor product of the Koszul resolution of ${\mathcal O}_{X'}$ with any locally free sheaf ${\mathcal O}_{\CP^n}(i)$ yields an exact sequence. Since the functor $\pi^G_*$ is exact and satisfies the projection formula as in \cite[Proposition 1.5]{Canonaco}, the sequence \eqref{E:star} must be exact. 

Together with Theorem \ref{T:bb}, the exactness of sequence \eqref{E:star} implies that
\begin{equation}\label{E:chi}
b_3 + b_7\; = \; \chi \left(X, \mathcal O( { \lfloor K_X/2 \rfloor }) \right) \; = \; \sum_{e=0}^{n-2}\; (-1)^e \;{{n-2} \choose e}\cdot \chi\left(\mathcal O_{\PP}\left(-ea + \lfloor K_X/ 2 \rfloor\right)\right).
\end{equation}
To identify the terms on the right hand side of \eqref{E:chi}, notice that cohomology of the sheaves $\mathcal O_{\PP}(k)$ vanish except possibly in degrees $0$ and $n$, and that 
\[
H^0 \left(\mathcal O_{\PP} \left(-ea + \lfloor K_X/2 \rfloor\right)\right) = 0\quad\text{for}\quad e\, \ge\, {(n-2)/2}.
\]
It follows from the Serre duality on $\PP$ (where $\mathcal O_{\PP}(K_{\PP}) = \mathcal O_{\PP} (-\sum q_i-1)$) that 
\[
H^n \left(\mathcal O_{\PP}\left(-ea + \lfloor K_X/2 \rfloor\right)\right) = H^0 \left(\mathcal O_{\PP}\left(((e-n+2)a + \lceil K_X/2 \rceil\right))\right).
\]
Now, calculate the Euler characteristics in formula \eqref{E:chi} in terms of only the non-vanishing cohomology groups and use the above isomorphism to express everything in terms of 0--cohomology. We obtain
\begin{multline*}
\chi \left(\mathcal O_{\PP}\left(-ea + \lfloor K_X/2 \rfloor\right)\right) 
= \sum_{e=0}^{\lfloor {{n-2} \over 2}\rfloor}(-1)^e{{n-2} \choose e} 
\big(\dim H^0 \left(\mathcal O_{\PP}\left(-ea + \lfloor K_X/2 \rfloor\right)\right) \\
+ \dim H^0 \left(\mathcal O_{\PP}\left(-ea + \lceil K_X/2 \rceil\right)\right)\big).
\end{multline*} 
Since the weights $a_1,\ldots, a_n$ are relatively prime, we can use the interpretation of the dimension of $H^0(\PP,\mathcal O(k))$ in terms of integer lattice points in a simplex to conclude that $\dim H^0 \left(\mathcal O_{\PP}\left(-ea + \lfloor K_X/2 \rfloor\right)\right)$ is the number of non-negative integer solutions $(x_1,\ldots, x_n, \allowbreak x_{n+1})$ of the equation
\[
\sum_{i=1}^n\;  x_i q_i\, +\, x_{n+1}\; =\;  - ea\, +\, \left\lfloor  \frac 1 2 \left( (n-2)\, a \, - \, \sum_{i=1}^n \;  q_i \, -\, 1\right)\right\rfloor,
\]
and similarly for $\dim H^0 \left(\mathcal O_{\PP}\left(-ea + \lceil K_X/2 \rceil\right)\right)$. One can easily check that, in both cases, the number of solutions equals the number of non-negative integer solutions $(x_1, \ldots, x_n)$ of the inequality  
\[
\sum_{i=1}^n\; \frac{x_i}{a_i}\, < \, -e\, +\, \frac 1 2 \, \left(n - 2 - \sum_{i=1}^n \; \frac 1 {a_i}\right)
\]
and, in particular, that 
\[
\dim H^0 \left(\mathcal O_{\PP}\left(-ea + \lfloor K_X/2 \rfloor\right)\right) = \dim H^0 \left(\mathcal O_{\PP}\left(-ea + \lceil K_X/2 \rceil\right)\right).
\] 
\noindent
With this understood, we wish to derive the formula of Theorem \ref{T:spectrum} from formula \eqref{E:chi}. This boils down to proving that $A = B$, where 
\medskip
\[
A \;=\; \sum_{e = 0}^s\; (e+1) \; \Big| A_e \Big|\quad\text{and}\quad B\; = \; \sum_{e = 0}^s \; (-1)^e\; {{n-2}\choose {e}} \; \Big| B_e \Big|,
\]
with
\smallskip
\[
A_e = \left\{ (x_1,\ldots,x_n)\;\Big |\; e \, +\, \sum_{i=1}^n \; \frac {x_i}{a_i}\; < \frac1 2\,\left(n - 2 - \sum_{i=1}^n \; \frac 1 {a_i}\right)\;\;\text{and}\;\; x_i < a_i\right\}
\]
and
\smallskip
\[
B_e = \left\{ (x_1,\ldots,x_n)\;\Big |\; e \, +\, \sum_{i=1}^n \; \frac {x_i}{a_i}\; < \frac1 2\,\left(n - 2 - \sum_{i=1}^n \; \frac 1 {a_i}\right)\right\}.
\]

\medskip\noindent
We begin by observing that any $n$--tuple $(x_1,\ldots, x_n) \in A_{s-k}$, which satisfies the conditions $0 \le x_i < a_i$, $i = 1,\ldots, n$, and 
\[
(s - k)\, +\, \sum_{i=1}^n \; \frac {x_i}{a_i}\; < \frac1 2\,\left(n - 2 - \sum_{i=1}^n \; \frac 1 {a_i}\right),
\]

\smallskip\noindent
gives rise to an $n$--tuple in $B_{s-p}$, $k = 0,\ldots, p$, by re-writing the above condition in the form 
\[
(s - p)\, +\, (p - k) \, +\, \sum_{i=1}^n \; \frac {x_i}{a_i}\; < \frac1 2\,\left(n - 2 - \sum_{i=1}^n \; \frac 1 {a_i}\right)
\]

\smallskip\noindent
and distributing the $p-k$ units into the summands $x_i/a_i$. The resulting $n$--tuple in $B_{s-p}$ is obtained by adding several copies of $a_i$ to the respective $x_i$. The number of ways to distribute $p - k$ units into $\ell$ slots, $\ell = 1,\ldots, p - k$, while accounting for their order, equals 
\smallskip
\[
{{p - k - 1}\choose {\ell - 1}} {{n}\choose {\ell}},
\]

\smallskip\noindent
where the first factor is the well known combinatorial quantity, the number of compositions of $p - k$ into $\ell$ parts. 

\smallskip

\begin{lemma}
\qquad\qquad $\displaystyle
\sum_{\ell = 1}^{p-k} \; {{p - k - 1}\choose {\ell - 1}} {{n}\choose {\ell}} \; = \; {{n + p - k - 1}\choose {n - 1}}$.
\end{lemma}

\begin{proof}
This follows by comparing the $t^{n-1}$ terms in the binomial expansions of the individual terms in the identity $(1 + x)^n \, (1 + x)^{p - k - 1} = (1 + x)^{n + p - k - 1}$.
\end{proof}

By adding all possible contributions of $A_{s - k}$ with $k = 0,\ldots, p$ into $B_{s - p}$ with $p = 0,\ldots, s$, we obtain the formula
\[
\Big| B_{s-p} \Big|\; = \; \sum_{k = 0}^p\; {{n + p - k - 1}\choose {n - 1}} \; \Big| A_{s-k} \Big|.
\]
Proving the formula $A = B$ of the theorem now reduces to checking the following combinatorial identity.

\begin{lemma}
\qquad\qquad $\displaystyle
\sum_{p=0}^k \; (-1)^p\, {{n + k - p + 1}\choose {n - 1}}{{n}\choose p} \; = \; k + 1$.
\end{lemma}

\begin{proof}
The proof will proceed by expanding the identity $(1 + x)^{-n-2} (1 + x)^n = (1 + x)^{-2}$ into a Taylor series in two different ways. On the one hand, 
\[
(1+x)^{-n-2}\; =\; \sum_{a=0}^{\infty}\; (-1)^a \, {{n+a+1}\choose {n + 1}} \, x^a
\]
and 
\[
(1+x)^n \; =\; \sum_{b=0}^n\; {n \choose  b}\, x^b
\]
hence
\[
(1 + x)^{-n-2} (1 + x)^n \; = \; \sum_{k=0}^{\infty} \; \sum_{a+b=k}\; (-1)^a\, {{n+a+1}\choose {n + 1}}{n \choose b}\,x^k.
\]

\medskip\noindent
Substituting $b = p$ and $a = k-p$ we turn the last expression into  
\smallskip
\[
\sum_{k=0}^{\infty}\; \sum_{p=0}^k\; (-1)^{k-p} \, {{n+k-p+1} \choose {n + 1}}{n \choose p}\, x^k.
\]

\smallskip\noindent
On the other hand, 
\[
(1+x)^{-2} \; = \; \sum_{k=0}^{\infty}\; (-1)^k\, (1+k) \, x^k.
\]

\smallskip\noindent
The result now follows by comparing the coefficients of the $x^k$ terms in the last two formulas.
\end{proof}


\section{The Seiberg--Witten equations}\label{S:sw}
This section builds on the observation that the right hand side of the formula of Theorem \ref{T:spectrum} matches a count of irreducible Seiberg--Witten monopoles on $\Sigma(a_1,\ldots, a_n)$; see Theorem \ref{T:moy}. This leads to an independent proof of the equality of the Casson invariants of $\Sigma(a_1,\ldots, a_n)$ in the Donaldson and Seiberg--Witten theories inspired by the Witten conjecture \cite{W}.

We begin with a brief account of the Seiberg--Witten theory following Mrowka, Ozsv{\'a}th and Yu \cite{MOY}. Given a closed oriented Riemannian spin 3-manifold $N$, the Seiberg--Witten equations are the system of non-linear partial differential equations 
\[
*F_A = \sigma(\psi), \quad D_A \psi = 0
\]
on a pair $(A,\psi)$, where $A$ is a $U(1)$--connection and $\psi$ a spinor. The solutions are called monopoles, and the solutions with $\psi \neq 0$ are called irreducible monopoles. The gauge equivalence classes of irreducible monopoles form the moduli space $\M(N)$. For a generic metric, and perhaps a small perturbation, the moduli space $\M(N)$ consists of finitely many points, which are canonically oriented. 

According to \cite{MOY}, the natural metric realizing the Thurston geometry on $\Sigma(a_1,\ldots, a_n)$ is generic for $n = 3$ and $n = 4$ but not for higher $n$. For the unique spin structure on $\Sigma(a_1,\ldots, a_n)$, the Seiberg--Witten moduli space $\M = \M(\Sigma(a_1,\ldots, a_n))$  is a collection of isolated points in the case of $n = 3$ and $n = 4$. In general, it consists of finitely many complex projective spaces of varying dimensions. Denote by $\#\M$ the oriented count of points in $\M$ after a small generic perturbation chosen so that each complex projective space contributes its Euler characteristic to $\#\M$. 

\begin{theorem}\label{T:moy}
For any complete intersection singularity \eqref{E:hamm}, instanton Floer homology of $\Sigma(a_1,\ldots, a_n)$ and the Seiberg--Witten monopoles on $\Sigma(a_1,\ldots,a_n)$ are related by the formula 
\[
b_3 + b_7\; = \; \#\M.
\]
\end{theorem}

\begin{proof} It follows from \cite{MOY} that
\[
\#\M\; = \; 2\, \sum_{e = 0}^s\; (e + 1)\, \big| A_e \big|.
\]
Now the claim follows from Theorem \ref{T:spectrum}.
\end{proof}

Let $\Sigma$ be an arbitrary oriented integral homology 3-sphere. For the unique spin structure on $\Sigma$, define the Seiberg--Witten analogue of the Casson invariant by the formula 
\[
\lambda_{\,\SW}(\Sigma)\; =\; -\,\#\M(\Sigma) -\, \frac 1 2\,\eta_{\Dir} (\Sigma)\, -\, \frac 1 8\,\eta_{\Sign} (\Sigma),
\]
where $\eta_{\Dir}$ and $\eta_{\Sign}$ are the eta--invariants of the spin Dirac and the odd signature operators, respectively; see Atiyah--Patodi--Singer \cite{APS:I}. Note that the individual terms in the definition of $\lambda_{\,\SW} (\Sigma)$ are metric dependent but their combination is a topological quantity \cite{Chen, Lim}. It was conjectured by Peter Kronheimer and proved by Yuhan Lim \cite{Lim} that 
\begin{equation}\label{E:lim}
\lambda(\Sigma)\; =\; \lambda_{\,\SW}(\Sigma)
\end{equation}
for all integral homology spheres $\Sigma$. In particular, equality \eqref{E:lim} holds in the special case of $\Sigma = \Sigma(a_1,\ldots, a_n)$, where it was also independently proved by Nicolaescu \cite[Formula (0.2)]{Nicolaescu} for $n = 3$ and $n = 4$. Our combinatorial formulas lead to yet another proof of identity \eqref{E:lim} for $\Sigma (a_1,\ldots, a_n)$. 

\begin{corollary}
For any complete intersection singularity \eqref{E:hamm}, we have 
\[
\lambda(\Sigma(a_1,\ldots,a_n))\; =\; \lambda_{\,\SW}(\Sigma(a_1,\ldots,a_n)).
\]
\end{corollary}

\begin{proof}
Combine the formula of Theorem \ref{T:moy} with formula \eqref{E:s2} to conclude that 
\[
\lambda(\Sigma(a_1,\ldots,a_n))\, = \, - \, \#\M (\Sigma(a_1,\ldots,a_n))\, +\, \bar\mu (\Sigma(a_1,\ldots,a_n))
\]
and complete the proof by using the formula 
\begin{equation}\label{E:rs}
\bar\mu (\Sigma(a_1,\ldots, a_n))\;=\; -\frac 1 2\,\eta_{\Dir} (\Sigma(a_1,\ldots, a_n)) - \frac 1 8\,\eta_{\Sign} (\Sigma(a_1,\ldots, a_n))
\end{equation}
of Ruberman and Saveliev \cite{RS}. Note that formula \eqref{E:rs} is a special case of a general formula for the Lefschetz number in monopole Floer homology proved by Lin, Ruberman, and Saveliev \cite{LRS}.
\end{proof}

\begin{remark}
It should be noted that both Lim's formula \eqref{E:lim} and the formula of \cite{LRS} (of which formula \eqref{E:rs} is a special case) are instances of a general Witten--style conjecture proposed by Mrowka, Ruberman, and Saveliev \cite{MRS} relating the Donaldson and Seiberg--Witten invariants of 4-manifolds $X$ with homology of $S^1 \times S^3$. Formula \eqref{E:lim} corresponds to $X = S^1 \times \Sigma$, and  formula \eqref{E:rs} to the mapping torus $X$ of the complex conjugation involution on the link $\Sigma (a_1,\ldots,a_n)$. The above proof of  \eqref{E:lim} then qualifies as an independent verification of the Witten--style conjecture of \cite{MRS} for $X = S^1 \times \Sigma(a_1,\ldots,a_n)$.
\end{remark}


\section{Comments and generalizations}\label{S:comments}
We have observed throughout this paper that dealing with Brieskorn homology spheres $\Sigma(p,q,r)$ is much easier than with the general case of $\Sigma(a_1,\ldots, a_n)$. One way to get around the complications arising when $n \ge 4$ is to use the following additivity in the instanton Floer homology \cite[Corollary 9]{Saveliev2}. Given a Seifert fibered homology sphere $\Sigma (a_1,\ldots,a_n)$, let $q = a_1\cdots a_j$ and $p = a_{j+1}\cdots a_n$ be the products of the first $j$, respectively, last $n-j$, Seifert invariants, for any $2 \le j \le n-2$. Then 
\begin{equation}\label{E:add}
I_* (\Sigma(a_1,\ldots,a_n))\,=\,I_* (\Sigma(a_1,\ldots,a_j,p))\,\oplus\,I_* (\Sigma(q,a_{j+1},\ldots,a_n)).
\end{equation}
After sufficiently many iterations, this formula reduces instanton homology calculation for $\Sigma(a_1,\ldots,a_n)$ to that for Brieskorn homology spheres.

The results of this paper can be partially generalized to more complicated links of singularities. Let $\Sigma$ be the link of an isolated complete intersection surface singularity and assume that $\Sigma$ is an integral homology sphere. As before, the Euler characteristic of $I_*(\Sigma)$ equals $2\lambda(\Sigma)$, and the Lefschetz number of the map induced on the instanton Floer homology by the complex conjugation involution equals $2\bar\mu(\Sigma)$ (see \cite{Taubes} and \cite[Section 9.2]{RS2}, respectively). In addition, the formula \eqref{E:fs} for the signature of the Milnor fiber, which was conjectured to hold for all links $\Sigma$ as above by Neumann and Wahl \cite{Neumann-Wahl}, has been proved for all links $\Sigma$ of splice type by N{\'e}methi and Okuma \cite{Nemethi-Okuma}. These results, however, are not sufficient for computing $I_*(\Sigma)$ because not much is known about the overall structure of $I_*(\Sigma)$ for general links.


\section{The $d$--invariant of torus knots}\label{S:knots}
Let $p$, $q$ be coprime positive integers and consider the singularity at zero of the complex polynomial $f(x,y) = x^p + y^q$. For any $t \in \Delta: = \{ t \in \mathbb{C} ~|~ |t| < \epsilon \ll 1 \}$, denote by $\mathcal{X}_t$ the compactification of the Milnor fiber $M(p,q) = f^{-1} (t)$ in the weighted projective space $\mathbb P = \CP^2(q:p:1)$. That is to say, $\mathcal{X}_t$ is the complex curve
\[
\mathcal{X}_t = \{ x^p+y^q - t z^{pq}=0 \}\; \subset\; \mathbb{P}.
\]
Note that, for any $t \in \Delta \setminus\{0\}$, the curve $\mathcal{X}_t$ is smooth of genus $(p-1)(q-1)/{2}$ since it intersects the line $z=0$ at a single point and the first Betti number of the affine curve $x^p+y^q=t$ is $(p-1)(q-1)$; see Milnor \cite{Milnor}.

\begin{remark}
The integer $pq-p-q-1$ is always even because $p$ and $q$ are coprime. Therefore, for each $\mathcal{X}_t$, we have a canonical theta characteristic 
\smallskip
\[
\mathcal O\Big(\frac{K_{\mathcal{X}_t}}{2}\Big) \, =\, \mathcal{O}_{\mathbb{P}}\Big(\frac{pq-p-q-1}{2}\Big)\Big|_{\mathcal{X}_t}
\]

\smallskip\noindent
obtained by the restricting the Weil divisor on $\PP$, computed using the adjunction formula, to the smooth divisor ${\mathcal{X}_t}$ as in Dolgachev \cite[Section 3.5.2]{Dolgachev}.
\end{remark}

The $d$--invariant of a knot $K \subset S^3$ was defined by Ozsv{\'a}th and Szab{\'o} \cite{OS} as
\[
d(K) = d(S^3_1(K)),
\] 
the correction term in Heegaard Floer homology of the 3-manifold $S^3_1(K)$ obtained by (+1)--surgery on the knot $K$. Let $K = T_{p,q}$ be the right-handed $(p,q)$--torus knot then
\[
S^3_1(T_{p,q}) = -\Sigma(p,q,pq-1).
\] 

\begin{theorem}\label{T:d}
For any $t \in \Delta \setminus \{0\}$, denote by $K_{\mathcal{X}_t}$ the canonical divisor of $\mathcal{X}_t$. Then
\begin{equation}\label{E:d=h}
d(T_{p,q})\; =\; -2 h^0 \left(\mathcal{X}_t, \mathcal O\left( \frac{K_{\mathcal{X}_t}} 2 \right)\right).
\end{equation}
\end{theorem}

\begin{proof}
Our proof will rely on two formulas of Borodzik and Nemethi \cite{BN}. One formula \cite[Corollary 7.3]{BN} states that
\medskip
\[
d(T_{p,q}) = -2 \left| \left\{ s \notin S_{p,q} ~\Big|~ s \geq \frac{(p-1)(q-1)}{2}  \right\} \right|
\]
for the semigroup
\[
S_{p,q} = \{ ap+bq \in \mathbb{Z} ~|~ a \geq 0,\; b \geq 0 \}.
\]
The other formula relates $S_{p,q}$ to the spectrum of the singularity of $f(x,y)$, 
\[
\Spec_f = \left\{\; \frac{i}{p}+\frac{j}{q} ~\Big|~ 0 < i < p,\; 0 < j < q\;\right\}.
\]
For any integer $0 \le \alpha < pq$, it claims that
\[
\Big| \{ s \notin S_{p,q} ~\big|~ s \geq \alpha \} \Big|\, =\, \left| \left\{ \left[1+ \frac{\alpha}{pq}, 2 \right) \cap \Spec_f \right\} \right|,
\]
see \cite[Proposition 7.1]{BN}. 

To calculate the right hand side of the formula \eqref{E:d=h}, we will follow an approach similar to that in Section \ref{S:3.3}. Consider the short exact sequence
\[
 0 \longrightarrow \mathcal O_{\mathbb P}(-pq) \longrightarrow \mathcal O_{\mathbb P} \longrightarrow \mathcal O_{\mathcal{X}_t} \longrightarrow 0. 
 \]
Note that $\mathcal O_{\mathbb P}(-pq)$ is locally free and $\mathcal {O}_{\mathbb P}\left((pq - p - q - 1)/2 \right)$ is locally free at $\mathcal{X}_t$; see Mori \cite{Mori} for details. Tensor the above short exact sequence with $\mathcal {O}_{\mathbb P}\left((pq - p - q - 1)/2\right)$ to obtain the short exact sequence
\begin{multline*}
0 \to \mathcal O_{\mathbb P}\left( \frac{-pq - p - q - 1} 2 \right) \to \mathcal O_{\mathbb P}\left( \frac{pq - p - q - 1} 2 \right) \to \mathcal O_{\mathcal{X}_t}\left( \frac{pq - p - q - 1} 2 \right) \to 0. 
\end{multline*}

\smallskip\noindent
It follows from Dolgachev \cite{Dolgachev} that 
\smallskip
\[
H^0\left(\mathbb{P},\,\mathcal O_{\mathbb P}\left( \frac{-pq - p - q - 1} 2 \right)\right) = H^1\left(\mathbb{P},\,\mathcal O_{\mathbb P}\left( \frac{-pq - p - q - 1} 2 \right)\right)=0
\]
and 
\smallskip
\[
\; H^1\left(\mathbb{P},\,\mathcal O_{\mathbb P}\left(\, \frac{pq - p - q - 1} 2 \right)\right)\; =\; H^2\left(\mathbb{P},\,\mathcal O_{\mathbb P}\left(\, \frac{pq - p - q - 1} 2 \right)\right)\, =\, 0.
\]

\smallskip\noindent
Therefore,
\begin{multline*}
H^0\left( \mathcal{X}_t, \mathcal O_{\mathcal{X}_t}\left( \frac{K_{\mathcal{X}_t}} 2 \right)\right) = 
H^0\left( \mathbb P, \mathcal O_{\mathbb P}\left( \frac{pq - p - q - 1} 2 \right)\right) \\ =\; \mathbb C\, \left\langle x^a y^b z^c ~ \Big| ~ qa+pb+c = \frac{pq-p-q-1} 2 \right\rangle,
\end{multline*}
from which we conclude that 
\[
h^0\left( \mathcal{X}_t, \mathcal O_{\mathcal{X}_t}\left( \frac{K_{\mathcal{X}_t}} 2 \right)\right)
\]
equals the number of integer lattice points $(a,b,c) \in \mathbb Z^3$ such that $a \ge 0$, $b \ge 0$, $c \ge 0$ and
\[
q a + p b + c\, =\, \frac{pq-p-q-1} 2.
\]
In turn, this equals the number of integer lattice points $(a,b) \in \mathbb Z^2$ such that $a \ge 0$, $b \ge 0$ and 
\[
 \frac{a+1}{p} +\, \frac{b+1}{q}\, \leq\, \frac{pq+p+q-1}{2pq}.
\]

\smallskip\noindent
Using the aforementioned Borodzik--Nemethi formulas and the symmetry of $\Spec_f$, we obtain
\begin{align*}
d(T_{p,q}) & = -2 \left| \left\{ s \notin S_{p,q} ~\Big|~ s \geq \frac{(p-1)(q-1)}{2}  \right\} \right| \\
& = -2 \left| \left\{ \left[1+ \frac{(p-1)(q-1)}{2pq}, 2 \right) \cap \Spec_f \right\} \right| \\
& = -2 \left| \left\{ \left( 0, \frac{pq+p+q-1}{2pq} \right] \cap \Spec_f \right\}\right| 
\; =\; -2 h^0 \left(\mathcal{X}_t, \mathcal O\left( \frac{K_{\mathcal{X}_t}} 2 \right)\right).
\end{align*}
\end{proof}

\begin{remark}
A modulo 2 version of the formula of Theorem \ref{T:d}, namely,
\smallskip
\begin{equation}\label{E:d=h mod 2}
- \frac 1 2\; d(T_{p,q})\; =\; h^0 \left(\mathcal{X}_t, \mathcal O\left( \frac{K_{\mathcal{X}_t}} 2 \right)\right) \pmod 2,
\end{equation}
can be derived more directly from the results already in the literature. According to \cite{OS}, 
 \[
 -\frac 1 2\; d(T_{p,q})\; = \;\sum_{j \ge 1}\; j a_j,
 \]
 where $a_j$ are the coefficients of the Alexander polynomial 
 \[
 \Delta (t)\; = \; a_0 + \sum_{j\ge 1}\; a_j (t^j + t^{-j})
 \]
of the knot $T_{p,q}$ normalized so that $\Delta (1) = 1$ and $\Delta(t^{-1}) = \Delta(t)$. A direct calculation then shows that 
 \[
 -\frac 1 2\; d(T_{p,q})\; = \;\sum_{j \ge 1}\; j a_j\; = \; \frac 1 2\; \Delta'' (1)\pmod 2,
 \]
 which is known to equal the Robertello invariant, or the Arf invariant, of the knot $T_{p,q}$. 
 
On the other hand, according to \cite[Proposition 4.2]{Libgober}, the Robertello invariant of $T_{p,q}$ is the difference of the Arf invariants of the $\ZZ/2$ quadratic forms corresponding to the theta characteristics (or spin structures) $L = (K_{\mathbb P}+{\X}_t)/2$ and $(K_{\mathbb P}+{\X}_0)/2$ on ${\X}_t$ and ${\X}_0$, respectively (since $\X_t$ contains no singularities of $\mathbb P$, the result of \cite{Libgober} applies even though $\mathbb P$ is singular). The claim now follows because $H^1 ({\X}_0,\ZZ/2) = 0$ and the Robertello invariant of the quadratic form corresponding to the theta characteristic $L$ equals 
\[
h^0 \left(\mathcal{X}_t, \mathcal O\left( \frac{K_{\mathcal{X}_t}} 2 \right)\right);
\]
see \cite[Proposition 2.1]{Libgober}.
\end{remark}

\begin{remark} 
One can use Riemann's singularity theorem \cite[Chapter 6, Section 1]{ACGH} to interpret the right-hand side of the relation \eqref{E:d=h mod 2} as the multiplicity of the theta divisor of the Jacobian of ${\mathcal X}_t$  at the point corresponding to (the translate of) canonical theta characteristic used in Theorem \ref{T:d}. This is a special case of the relation \cite{Libgober} between the mod 2 Seifert form of a link of a plane curve singularity and the canonical theta characteristic on the Steenbrink Jacobian of the Milnor fiber of singularity. One may wonder is this observation, together with the above results, extends to other links of singularities.
\end{remark}


\section{A formula for the $\bar\mu$--invariant}\label{S:mubar}
Let $X$ be the compactification of the Milnor fiber of the complete intersection \eqref{E:hamm} in the weighted projective space and let $K$ be the canonical divisor of $X$. In this section, we will prove the formula
\[
b_3 +  b_7\; = \;\chi \left(X, \mathcal O\left(\left\lceil {K}/2 \right\rceil\right) \right)\; =\;\chi \left(X, \mathcal O \left(\left\lfloor {K}/2 \right\rfloor\right) \right)
\]
of Theorem \ref{T:bb}. As in the proof of Theorem \ref{T:euler}, we will first compute $\chi \left(X, \mathcal O\left(\left\lceil {K}/2 \right\rceil\right) \right)$ using the orbifold version of the Atiyah--Singer index theorem due to Kawasaki \cite{Kawasaki} and then compare it with $b_3 + b_7$. 

Our task is significantly simplified by the fact that the necessary index calculations have already been done by Fukumoto and Furuta \cite[Section 5]{FF}. Their formulas easily imply that $\chi \left(X, \mathcal O\left(\left\lceil {K}/2 \right\rceil\right) \right) = \chi \left(X, \mathcal O \left(\left\lfloor {K}/2 \right\rfloor\right) \right)$ and reduce the proof to checking that $\bar\mu (\Sigma(a_1,\ldots,a_n))$ is given by the formula of \cite[Theorem 12]{FF} for the appropriate choice of the orbifold line bundle. In the case of even $a_1\cdots a_n$ this was done by Fukumoto, Furuta and Ue \cite[Theorem 2]{FFU}, see also Saveliev \cite[Theorem 1]{Saveliev3}, hence we only need to address the case of odd $a_1\cdots a_n$. After accounting for varying orientation conventions, we wish to verify the formula
\[
\begin{aligned}
\bar\mu(\Sigma(a_1,\ldots,a_n)) = &\; \frac 1 {8a_1\cdots a_n}  - \frac1 8 +\; \frac1 8\; \sum_{n=1}^n\; 
\frac 1{a_i}\; \sum_{\ell=1}^{a_i-1}\;\cot\left(\frac{\pi\ell}{a_i}\right)\cot\left(\frac{\pi b_i \ell}{a_i}\right) 
\\[7pt] 
+ \;\frac1 4\;  \sum_{n=1}^n\; &
\frac 1 {a_i}\;\sum_{\ell = 1}^{a_i - 1}\; \cos \left(\frac {\pi (1 + b_i + 2m b_i)\ell}{a_i}\right)\csc\left(\frac{\pi \ell}{a_i}\right)\csc\left(\frac{\pi b_i \ell}{a_i}\right),
\end{aligned}
\]
where 
\[
m\, =\, -\frac {a_1\cdots a_n} 2\cdot  \left(-2\, +\, \sum_{i=1}^n\; \left(1 - \frac 1{a_i}\right)\right)\quad\text{and}\quad
a_1\cdots a_n\cdot \sum_{i=1}^n\; \frac {b_i}{a_i}\, =\, -1.
\]
After checking that $(1 + 2m b_i)/a_i$ is an odd integer for all $i$ we reduce the formula to
\[
\begin{aligned}
\bar\mu(\Sigma(a_1,\ldots,a_n)) = \; \frac 1 {8a_1\cdots a_n}  - \frac1 8 + \; \frac1 8\; & \sum_{i=1}^n\; 
\frac 1{a_i}\; \sum_{\ell=1}^{a_i-1}\;\cot\left(\frac{\pi\ell}{a_i}\right)\cot\left(\frac{\pi b_i \ell}{a_i}\right) 
\\[7pt] 
+ \;\frac1 4\;  \sum_{i=1}^n\; &
\frac 1 {a_i}\;\sum_{\ell = 1}^{a_i - 1}\; (-1)^{\ell}\csc\left(\frac{\pi \ell}{a_i}\right)\cot\left(\frac{\pi b_i \ell}{a_i}\right).
\end{aligned}
\]

\smallskip\noindent We will proceed by identifying this formula with the known formula for the $\bar\mu$--invariant,
\[
\begin{aligned}
\bar\mu(\Sigma(a_1,\ldots,a_n))\; =\; 
\frac 1 {8 a_1\cdots a_n} - \frac 1 8  - 
\frac 1 2\; & \sum_{i=1}^n\; s(a_1\cdots a_n/a_i, a_i) \\ 
-\, & \sum_{i=1}^n\; s(a_1\ldots a_n/a_i,a_i; 1/2,1/2),
\end{aligned}
\]
see \cite[Formula 4]{RS}, which is based of a calculation by Nicolaescu \cite{Nicolaescu}. All we need to do is match trigonometric sums in the former formula with the respective Dedekind--Rademacher sums in the latter. 

Recall that the Dedekind--Rademacher sums $s(q,p;x,y)$ are defined for pairwise coprime integers $p, q > 0$ and arbitrary real numbers $x$ and $y$ by the formula
\[
s(q,p;x,y)\;= \sum_{h\mmod p} \; \lt \frac {h + y} p \rt \lt \frac {q(h + y)} p + x \rt,
\]
where, for any real number $r$, we set
\[
(\hspace{-0.02in}( r )\hspace{-0.02in}) = 
\begin{cases}
\; 0, &\text{if $r \in \mathbb Z$}, \\
r - \lfloor r \rfloor - 1/2, &\text{if $r \notin \mathbb Z$}.
\end{cases}
\]
When both $x$ and $y$ are integers, we get back the classical Dedekind sum 
\[
s(q,p)\;= \sum_{h\mmod p} \; \lt \frac h p \rt \lt \frac {qh} p \rt.
\]

\begin{lemma}
Let $h$, $p$ be coprime integers such that $p \ge 1$ is odd and let $\zeta = e^{2\pi i/p}$. Then 
\begin{equation}\label{E:rg1}
\lt \frac h p \rt \; =\; \frac i {2p}\;\sum_{n=1}^{p-1}\; \cot\left(\frac {\pi n} p\right) \zeta^{hn}\quad\text{and}
\end{equation}
\begin{equation}\label{E:rg2}
\lt \frac h p + \frac 1 2 \rt \, =\, \frac i {2p}\;\sum_{n=1}^{p-1}\; (-1)^n \csc\left(\frac {\pi n} p\right) \zeta^{hn}.
\end{equation}
\end{lemma}

\begin{proof}
The first formula is proved on pages 13--14 of \cite{RG}, and we will prove the second by a slight modification of that argument. The function in question has period $p$ and thus can be expressed as
\[
\lt \frac h p + \frac 1 2 \rt \; =\; \sum_{m \mmod p}\; a_m \zeta^{mh}
\]
with certain Fourier coefficients $a_m$ which can be found from the formula
\[
\sum_{h \mmod p}\; \lt \frac h p + \frac 1 2\rt \zeta^{-hn}\;=\;\sum_{m \mmod p}\; a_m\sum_{h \mmod p}\;\zeta^{(m-n)h}\;=\; p a_n.
\]
According to Lemma 1 on page 4 of \cite{RG},
\[
\sum_{h \mmod p}\; \lt \frac {h+x} p \rt\, =\,(\hspace{-0.02in}( x )\hspace{-0.02in}),
\]
hence for $n = 0$ we have 
\[
a_0 = \frac 1 p \sum_{h \mmod p}\; \lt \frac h p + \frac 1 2 \rt\, =\, \frac 1 p \sum_{h \mmod p}\; \lt \frac {h + p/2} p \rt\, =\, \frac 1 p\, \left(\hspace{-0.04in}\left( \frac p 2 \right)\hspace{-0.04in}\right) = 0.
\]
On the other hand, if $n \ne 0$, we obtain
\begin{multline*}
a_n\, =\, \frac 1 p\; \sum_{h=1}^{p-1}\; \lt \frac h p + \frac 1 2 \rt \zeta^{-hn} 
\,=\, \frac 1 {p^2} \sum_{h=1}^{(p-1)/2}\; h\zeta^{-hn}\, + \\
\frac 1 {p^2} \sum_{h=(p+1)/2}^{p-1}\; (h - p)\zeta^{-hn} 
\,=\, \frac 1 {p^2}\; \sum_{h=1}^{p-1}\; h\zeta^{-hn}
- \frac 1 p \sum_{h=(p+1)/2}^{p-1}\; \zeta^{-hn}.
 \end{multline*}
A straightforward calculation with the geometric series shows that 
\[
\sum_{h=1}^{p-1}\; h\zeta^{-hn} = - \frac p {1 - \zeta^{-n}}\quad\text{and}\;
\sum_{h=(p+1)/2}^{p-1}\; \zeta^{-hn} = \frac {\zeta^{-n(p+1)/2} - 1}{1 - \zeta^{-n}}
\]
Therefore,
\[
a_n = -\frac 1 p\cdot \frac {\zeta^{-n(p+1)/2}}{1 - \zeta^{-n}}
\]
and
\begin{multline*}
\lt \frac h p + \frac 1 2 \rt \; =\;-\frac 1 p\; \sum_{n=1}^{p-1}\; \frac {\zeta^{-n(p+1)/2}}{1 - \zeta^{-n}} \, \zeta^{hn}  \\
= -\frac 1 p\; \sum_{n=1}^{p-1}\; \frac {\zeta^{-np/2}}{\zeta^{\,n/2} - \zeta^{-n/2}} \, \zeta^{hn} 
\; =\; \frac i {2p}\,\sum_{n=1}^{p-1}\; (-1)^n \csc\left(\frac {\pi n} p\right) \zeta^{hn}.
\end{multline*}
\end{proof}

\begin{proposition}
For any $i = 1,\ldots,n$, 
\[
-\frac 1 2\, s(a_1\cdots a_n/a_i,a_i)\; =\; \frac1 {8 a_i}\; \sum_{\ell=1}^{a_i-1}\;\cot\left(\frac{\pi\ell}{a_i}\right)\cot\left(\frac{\pi b_i \ell}{a_i}\right)
\]
\end{proposition}

\begin{proof}
It follows from \cite[Lemma 1]{RS} that $- s(a_1\cdots a_n/a_i,a_i) = s(b_i,a_i)$. The result now follows from formula (26) of \cite{RG}, which is an easy consequence of \eqref{E:rg1}.
\end{proof}

\begin{proposition}
For any $i = 1,\ldots,n$, 
\[
- s(a_1\cdots a_n/a_i,a_i;1/2,1/2)\; =\; \frac1 {4a_i}\;\sum_{\ell = 1}^{a_i - 1}\; (-1)^{\ell}\csc\left(\frac{\pi \ell}{a_i}\right)\cot\left(\frac{\pi b_i \ell}{a_i}\right).
\]
\end{proposition}

\begin{proof}
Lemma 3 of \cite{RS} implies that $- s(a_1\cdots a_n/a_i,a_i) = s(b_i,a_i; 1/2, 0)$. The result now follows by a straightforward calculation using formulas \eqref{E:rg1} and \eqref{E:rg2}:
\begin{align*}
s(b_i,a_i; 1/2, 0) & = \sum_{h\mmod a_i} \; \lt \frac h {a_i} \rt \lt \frac {h b_i} {a_i} + \frac 1 2 \rt \\
& = - \frac 1 {4a_i^2}\; \sum_{m, n=1}^{a_i-1}\;(-1)^m \cot\left(\frac {\pi n}{a_i}\right)\csc\left(\frac {\pi m}{a_i}\right) \sum_{h \mmod a_i} \zeta^{h(n+m b_i)}.
\end{align*}

\smallskip\noindent
The sum over $h \mmod a_i$ vanishes, unless $n + m b_i = 0 \mmod a_i$, when it has value $a_i$, which gives the desired formula.
\end{proof}


\end{document}